\documentclass[11pt, notitlepage]{article}   

\usepackage{amsmath,amsthm,amsfonts}   

\usepackage{amscd}
\usepackage{color}      
\usepackage{amssymb}
\usepackage{graphicx}

\usepackage{enumerate}
\usepackage{amsmath,amsfonts,latexsym,float}
\usepackage{array,amsthm,amssymb,graphicx,epstopdf, tikz}
\usepackage{epsfig}
\usepackage{mathdots}

\theoremstyle{plain}
\newtheorem{theorem}{Theorem}[section]
\newtheorem{lemma}[theorem]{Lemma}
\newtheorem{corollary}[theorem]{Corollary}
\newtheorem{proposition}[theorem]{Proposition}
\newtheorem{porism}[theorem]{Porism}

\newtheorem{observation}[theorem]{Observation}

\newtheorem{remark}[theorem]{Remark}

\newtheorem{example}[theorem]{Example}

\theoremstyle{definition}
\newtheorem{definition}[theorem]{Definition}

\def\finf{\mathop{{\rm I}\kern -.27 em {\rm F}}\nolimits}

\def\l{\lambda}

\def\Z{\Bbb Z}

\def\N{\Bbb N}

\def\({\left(}
\def\){\right)}
\def\l{\left}
\def\r{\right}

\def\supp{\text{supp }}
\newcommand{\rmv}[1]{}

\begin{document}

\title{On Laplacian Monopoles}

\author{{\bf{Cong X. Kang}}$^1$, {\bf{Gretchen L. Matthews}}$^2$ and {\bf{Justin D. Peachey}}$^3$\\
\small Texas A\&M University at Galveston, Galveston, TX 77553, USA$^1$\\
\small Virginia Tech, Blacksburg, VA 24061, USA$^2$\\
\small Arlington, VA 22204, USA$^3$ \\
{\small\em kangc@tamug.edu}$^1$; {\small\em gmatthews@vt.edu}$^2$; {\small\em jdpeachey@gmail.com}$^3$}
\maketitle
\date{}

\begin{abstract}
We consider the action of the (combinatorial) Laplacian of a finite and simple graph on integer vectors. By a \emph{Laplacian monopole} we mean an image vector negative at exactly one coordinate associated with a vertex. We consider a numerical semigroup $H_f(P)$ given by all monopoles at a vertex of a graph. The well-known analogy between finite graphs and algebraic curves (Riemann surfaces) has motivated much work. More specifically for us, the motivation arises out of the classical Weierstrass semigroup of a rational point on a curve whose properties are tied to the Riemann-Roch Theorem, as well as out of the graph theoretic Riemann-Roch Theorem demonstrated by Baker and Norine. We determine $H_f(P)$ for some families of graphs and demonstrate a connection between $H_f(P)$ and the vertex (also edge) connectivity of a graph. We also study $H_r(P)$, another numerical semigroup which arises out of the result of Baker and Norine, and explore its connection to $H_f(P)$ on graphs. We show that $H_r(P)\subseteq H_f(P)$ in a number of special cases. In contrast to the situation in the classical setting, we demonstrate that $H_f(P)\setminus H_r(P)$ can be arbitrarily large and identify a potential obstruction to the inclusion of $H_r(P)$ in $H_f(P)$ in general, though we still conjecture this inclusion. We conclude with a few open questions.

\end{abstract}


\section{Introduction}
\textbf{1.1. Overview.} Let $G = (V(G),E(G))$ be a finite, simple, and undirected graph of order $n\geq 2$. The (combinatorial) Laplacian of $G$ is the $n\times n$ square matrix $\Delta=D-A$, where $D$ is the diagonal matrix of degrees and $A$ is the adjacency matrix of $G$. The Laplacian of graphs is well-studied; see~\cite{Biggs-book,Bollobas,Lorenzini} for examples. Consider $\Delta$ as a linear map from $\mathbb{Z}^n$ to $\mathbb{Z}^n$. Letting the summands of $\mathbb{Z}^n$ be indexed by the vertices $V(G)$ of $G$, we identify the domain of $\Delta$ with integer-valued functions on $V(G)$. It immediately follows from the definition of $\Delta$ that $\sum_{v\in V(G)}(\Delta(f))_v=0$ for any $f\in \mathbb{Z}^n$. Fix $P\in V(G)$, and let $\mathcal{F}(P)=\{f\in\mathbb{Z}^n: (\Delta(f))_v\geq 0 \mbox{\ if\ }v\neq P \}$. Thus, $\mathcal{F}(P)$ is the set of $P$-monopoles, namely functions having $P$ as their only \emph{pole}. The set $H_f(P)=\{\alpha\in \mathbb{N}: \exists f\in \mathcal{\mathcal{F}(P)} \mbox{\ such that \ }(\Delta(f))_P=-\alpha \}$ is the primary object of study in this paper. \\

Why are we interested in $H_f(P)$? Some more notions and terminology are needed. We will identify $\mathbb{Z}^n$ of the codomain of $\Delta$ with the free abelian group $Div(G)$ on $V(G)$. Elements of $Div(G)$ are called \emph{divisors}, and we write $D=\sum_{i=1}^{n}a_iP_i$, where $P_i\in V(G)$, for a divisor $D\in Div(G)\cong \mathbb{Z}^n$. A divisor $D$ is called \emph{effective} if $a_i\geq 0$ for each $i\in\{1,\ldots, n\}$. Define the support of a divisor $D$ by $\supp (D)=\{P_i\in V(G): a_i\neq 0\}$. Clearly, for any divisor $D$ we can write $D=A-B$, where $A$ and $B$ are both effective, and $\supp(A)\cap\supp(B)=\emptyset$; in this case, $B$ is said to be the \emph{polar part} of $D$, and $P\in\supp(B)$ is called a \emph{pole}. Expressing the divisor of a function $f$ on $V(G)$ by writing $\Delta(f)=A-B$, with $A$ and $B$ both effective and having trivial intersection in their supports, we thus call $B$ the \emph{polar divisor} of $f$. \\

As is widely known, there is an analogy between finite graphs and algebraic curves (Riemann surfaces) -- indeed, Chung wrote a monograph \cite{Chung}, where the (analytic) Laplacian is featured, pursuant to an analogy between spectral graph theory and spectral Riemannian geometry. Now, the $H_f(P)$ defined above is the graph theoretic analogue of the (classical) Weierstrass semigroup of a rational point on a curve, and its study dates back to the work of Hurwitz \cite{Hurwitz}. Given a nonsingular curve $X$ of genus $g$ defined over a base field, the classical Weierstrass semigroup $\mathfrak{H}_f(P)$ at $P$ is the collection of nonnegative integers $\alpha$ such that there exists a function $f$ on $X$ with polar divisor exactly $\alpha P$. The classical $\mathfrak{H}_f(P)$ is well-known to satisfy the following properties:
\begin{enumerate}
\item $\mathfrak{H}_f(P)$ is a numerical semigroup; i.e., it is a subset of the set of nonnegative integers $\N$, and it is closed under addition, contains the zero element, and has a finite complement in $\N$;
\item $\N \setminus \mathfrak{H}_f(P) \subseteq [0,2g-1]$;
\item $|\N \setminus \mathfrak{H}_f(P)|=g$;
\item $\alpha \in \mathfrak{H}_f(P)$ if and only if $\alpha\in\mathfrak{H}_r(P)=\{\beta\in\N:\ell(\beta P) = \ell ((\beta - 1)P) +1\}$, where $\ell(\beta P)$ denotes the dimension of the space of rational functions on $X$ having a pole only at $P$ and of order at most $\beta$.
\end{enumerate}

For more on classical Weierstrass semigroups, please see \cite{arbarello, survey}. Returning to graphs, in the present terminology, $H_f(P)$ is again the collection of nonnegative integers $\alpha$ that arise as coefficients of the polar divisors of functions (defined on $V(G)$) with a single pole at $P$; it is easily seen that it is a numerical semigroup. In \cite{BN}, Baker and Norine defined the rank $r(D)$ of a divisor $D$ on a graph $G$, and they proved that the rank function satisfies a Riemann-Roch-type theorem; they carried further their investigation in \cite{B,BNhyperelliptic}. The Riemann-Roch theorem on graphs by Baker and Norine sets up a study of Weierstrass semigroups on graphs in parallel with that in the classical setting. We determine $H_f(P)$ for trees, unicyclic graphs, and complete graphs. We demonstrate a connection between $H_f(P)$ and the vertex (also edge) connectivity of a graph. Denote the analogue of $\mathfrak{H}_r(P)$ in graphs by $H_r(P)$. We also study $H_r(P)$ and explore its connection to $H_f(P)$ on graphs. We show that $H_r(P)\subseteq H_f(P)$ in a number of special cases, including for a vertex $P$ of degree-one or having a neighbor of degree-one on a graph in general, any vertex $P$ on a tree, and on a unicyclic graph with some mild symmetry requirement at the vertex $P$. We identify a potential obstruction to the inclusion of $H_r(P)$ in $H_f(P)$ in general, though we are still inclined to conjecture this inclusion. In contrast to the equivalence given in claim 4 above, we demonstrate that $H_f(P) \setminus H_r(P)$ can be arbitrarily large. We conclude with a few open questions.\\

Here is a further bit of context and justification for the present line of inquiry. If we consider the Laplacian $\Delta$ of a graph of order $n$ as a vector space map over a field $\mathbb{F}$, then the image (or range) of $\Delta$ is simply the orthogonal complement of the all-one vector in $\mathbb{F}^n$, and this would be the end of the story. The action of $\Delta$ on $\Z^n$, however, has given rise to the large, rich, and interesting theory of the Jacobian (the torsion part of $\frac{\Z^n}{\Delta(\Z^n)}$) of a finite graph. We will say a bit more about the Jacobian herein later; see \cite{Bacher et al.,Biggs-Potential Theory,Biggs-Chip Firing,Cori-Rossin, Gsil-Royle,Lorenzini} for details. It is not a priori clear how one may determine the image of $\Delta$ from its cokernel. For example, a unicyclic graph $G$ has Jacobian isomorphic to $\frac{\Z}{k\Z}$, where $k$ is the length of the unique cycle in $G$, by the matrix tree theorem and~\cite{Lorenzini}. However, as we'll see in Section~\ref{exs}, $\Delta(\Z^n)$ is not determined by $k$. There are discussions on the structure of $\Delta(\Z^n)$ (the Laplacian lattice) to be found in the existing literature; see~\cite{Amini-Manjunath, Manjunath} for examples.\\

\textbf{1.2. More Notation and Basic Facts.} We will call $g=|E(G)|-|V(G)|+1$ the \emph{genus} of the graph $G$, in conformity with~\cite{BN}. We are mindful that $g$ is commonly known as the ``cycle rank" or the ``cyclomatic number" in graph theory literature wherein the term genus refers to a different notion. As explained in~\cite{BN} and also evident in this paper, calling $g$ the genus is a recognition of the fact that it plays a role for graphs analogous to what genus (the number of holes) plays for Riemann surfaces or algebraic curves; such thinking should be helpful in understanding the present subject matter.\\

The set of positive integers is denoted by $\Z^+$. Given $a_1, \dots, a_k \in \Z^+$ with $\gcd(a_1, \dots, a_k)=1$, the numerical semigroup generated by $\{a_1, \dots, a_k\}$ is $\l< a_1, \dots, a_k \r>=\{\sum_{i=1}^k c_i a_i, c_i \in \N\}$. We say that $\alpha$ is a gap of $\l< a_1, \dots, a_k \r>$ (equivalently, of a numerical semigroup $S$) if and only if $\alpha \in \N \setminus \l< a_1, \dots, a_k \r>$ (equivalently, $\alpha \in \N \setminus S$). A general reference for numerical semigroups is \cite{ns}. The neighborhood of a vertex $v$ is $N(v)=\{w\in V(G):wv\in E(G)\}$.\\

The degree of a divisor $D=\sum_{i=1}^n a_i P_i$ is $\deg(D) = \sum_{i=1}^n a_i$. Given $A, B\in Div(G)$, we declare $B\geq A$ if $B-A$ is effective; this imposes a partial order on $Div(G)$. The subset of all effective divisors of degree $k$ is denoted by $Div_{+}^k(G) = \{D \in Div(G) : D \geq 0 \text{\ and\ } \deg(D) = k\}$.  \\

Denote the set of integer-valued functions $f$ on $V(G)$ by $\mathcal{M}(G)$; thus $\mathcal{M}(G)\cong\Z^n$. It can be easily checked from the definition of the Laplacian that the \emph{divisor of} $f$, namely $\Delta(f)$, can be more explicitly expressed as follows:

$$\Delta(f)=\sum_{v \in V(G)}\left(\sum_{w \in N(v)} (f(v) - f(w))\right)v = \sum_{v \in V(G)}\left(f(v) \deg v - \sum_{w \in N(v)} f(w) \right) v.$$ \\

We will write $\Delta(f)= \sum_{v\in V(G)}\Delta_v(f)\,v$ and thus indicate the coefficient of a vertex $v$ in $\Delta(f)$ by $\Delta_v(f)$. We say that $f$ has a pole at $v$ (resp., zero at $v$) if $\Delta_v(f) < 0$ (resp.,  $\Delta_v(f) > 0$). A divisor $D$ is called \emph{principal} if $D=\Delta(f)$ for some $f\in \mathcal{M}(G)$.

The \emph{Jacobian} of a graph $G$ is the quotient group $Jac(G)=\frac{Div^0(G)}{Prin(G)}$, where $Div^0(G)$ denotes the set of all divisors of degree zero on $G$ and $Prin(G)$ denotes the set of all principal divisors on $G$. It is not hard to see $\frac{\Z^n}{\Delta(\Z^n)}\cong \Z\oplus Jac(G)$. The Jacobian of a graph has also been variously referred to as the \emph{Picard group}, the \emph{critical group}, or the \emph{sandpile group}; it is well-studied; see above citations.\\

The rest of this paper is organized as follows. Section \ref{main} contains the more general results. Examples are given in Section \ref{exs}. The paper concludes with further discussion and open questions in Section \ref{conclusion}.

\section{Weierstrass semigroups on graphs} \label{main}

We begin this section with a useful observation. The linearity of the Laplacian and the fact that a function $h$ constant on $V(G)$ has $\Delta(h)=0$ implies $\Delta(f+a)=\Delta(f)$ for any $a\in\Z$, where $(f+a)(v)=f(v)+a$ for $v\in V(G)$. Taking $a=\max \left\{ f(v): v \in V(G) \right\}$, we have $(f-a)(v)=f(v)-a \leq 0$.

\begin{observation} \label{fact}
Given $f, h \in \mathcal M(G)$, $f+h \in \mathcal{M}(G)$ and $\Delta(f+h)=\Delta(f)+\Delta(h)$. In particular, $\Delta(f+a)=\Delta(f)$ for any $a\in\Z$, and we can choose $b\in\Z$ such that $(f+b)(v)\leq 0\ \forall v\in V(G)$.
\end{observation}

Next, we check that $H_f(P)$ is what is customarily referred to as a numerical semigroup -- though it obviously contains zero and is thus a monoid.

\begin{proposition}\label{semigp}
Let $P$ be a vertex of a graph $G$. If $\alpha, \beta \in H_f(P)$, then $\alpha + \beta \in H_f(P)$.
\end{proposition}

\begin{proof}
If $\alpha, \beta \in H_f(P)$, there exist $f, h \in \mathcal M(G)$ such that $\Delta(f) = A - \alpha P$ and $\Delta(h) = B - \beta P$ where $A,B \geq 0$ and $P \notin \supp(A)\cup\supp(B)$. By Observation~\ref{fact}, $\Delta(f+h) = A+B-\left( \alpha + \beta \right)P$, and we have $\alpha + \beta \in H_f(P).$
\end{proof}

Now, noting $H_f(P) = \left\{ \alpha \in \mathbb{N} : \exists f \in  \mathcal M(G) \text{ with } 																							\Delta(f) = A - \alpha P, A \geq 0 \text{ and } P \notin \text{supp} A \right\}$, we call $H_f(P)$ the Weierstrass semigroup at the vertex $P$, in alignment with the theory on algebraic curves. To aid our investigation of $H_f(P)$, we consider another set $H_r(P)$, as well as some further notions defined in \cite{BN}.\\

Let $S(D)=\{k \in \N :\forall E\in Div_{+}^k(G)\;\exists f\in\mathcal{M}(G)\text{\ such\ that\ }D+\Delta(f)\geq E\}$; the \emph{rank} $r(D)$ of the divisor $D$ is $\max S(D)$ if $S(D)\neq \emptyset$, and $r(D)=-1$ otherwise. Then, define

\[H_r(P) = \left\{ \alpha \in \mathbb{N}\ :\ r(\alpha P) = r((\alpha-1)P) + 1\right\}.\]

In \cite[Theorem 1.12]{BN}, Baker and Norine prove the following analogue of the Riemann-Roch Theorem for finite graphs: $$r(D) = \deg(D) + 1 - g + r(K-D), \hspace{1in} \mbox{(R-R)}$$ where $K=\sum_{v\in V(G)}(\deg(v)-2)v$ is called the \emph{canonical divisor} on $G$. Notice that $\deg(K)=2|E(G)|-2|V(G)|=2g-2$. It follows immediately from the definition of $r(D)$ that $r(D) =-1$ if $\deg D<0$, as a divisor of negative degree cannot be greater or equal to a divisor of nonnegative degree. Notice from (R-R) that if $A \leq B$, then $r(B)-r(A) \leq \deg(B-A)$ and, in particular, we have $r(\alpha P)-r((\alpha-1)P)\leq 1$.

We will show that $\N \setminus H_r(P) \subseteq [0,2g-1]$, and $|\N \setminus H_r(P)|=g$. Moreover, we will show $|H_f(P) \setminus H_r(P)|$ can be arbitrarily large. Let $G_f(P)=\N \setminus H_f(P)$ and $G_r(P)=\N \setminus H_r(P)$.

\begin{lemma}\label{hrgaps}
Let $G$ be a graph, and let $P \in V(G)$.  If $\alpha \geq 2g$, then $\alpha \in H_r(P)$.  Hence, $G_r(P)$ is finite. In fact, $|G_r(P)| = g$ and $G_r(P) \subseteq [0,2g-1]$.
\end{lemma}

\begin{proof}
If $\alpha \geq 2g$, then $\deg(K-(\alpha-1)P)=2g-2-(\alpha-1)\leq 2g-2-(2g-1)=-1$. Thus $r(K-(\alpha-1)P)=-1=r(K-\alpha P)$, and we have
$$r(\alpha P)=\alpha+1-g+(-1)=\alpha-g > \left(\alpha-1\right)+1-g+(-1)=r \left( (\alpha-1)P \right).$$
Hence, $\alpha \in H_r(P)$ and $G_r(P)  \subseteq [0,2g-1]$. To see that $|G_r(P)| = g$,
notice that $$0=r(0P) \leq r(P) \leq r(2P) \leq r \left((2g-1)P \right)=g-1.$$
\end{proof}

Given a divisor $D$ with $r(D)=k-1$, define $\mbox{Obstr}(D)=\{A\in Div_{+}^{k}(G):\Delta(f)+D-A\not\geq 0\;\forall f\in \mathcal{M}(G)\}$. For a divisor $D=\alpha P$ with $r(\alpha P)=k-1$, also define $\mbox{Obstr}_0(\alpha P)=\{A\in\mbox{Obstr}(\alpha P):\supp(A)\not\ni P\}$. Whence we define $$H_{r}^{red}(P) = \{ \alpha \in \mathbb{N}\ :\ r(\alpha P) = r((\alpha-1)P) + 1 \mbox{ and } \mbox{Obstr}_0((\alpha-1)P)\neq \emptyset\}.$$

The following result establishes the relationship between $H_{r}^{red}(P)$ and $H_f(P)$.

\begin{theorem} \label{containment}
Let $G$ be a graph and $P$ be a vertex of $G$.  Then \[H_r^{red}(P) \subseteq H_f(P).\]  
\end{theorem}

\begin{proof}
Let $\alpha \in H_r^{red}(P)$.  Then $r((\alpha - 1)P) = k-1$ and $r(\alpha P) = k$ for some $k \in \Z$. It follows that there exists $A\in \mbox{Obstr}_0((\alpha-1)P)$ so that for all $f \in \mathcal{M}(G)$, $$(\alpha - 1) P + \Delta(f)-A \not\geq 0.$$ Since $r(\alpha P)=k$, there exists $h \in \mathcal{M}(G)$ so that $$\alpha P +\Delta(h)- A \geq 0\ \ (*); \mbox{ yet}$$
$$(\alpha - 1) P + \Delta(h)- A \not\geq 0\ \ (**).$$
Inequality (*) implies that $\Delta_Q(h)-A_Q\geq 0\ \forall Q\neq P$; which, along with the fact $P\not\in \supp(A)$, forces inequality (**) to imply $\Delta_P(h)+(\alpha-1)<0\ (\dag)$. Inequality (*), together with $P\not\in \supp(A)$, also implies $\Delta_P(h)+\alpha \geq 0\ (\ddag)$. Now, inequalities $\dag$ and $\ddag$ together imply that $-\Delta_P(h)=\alpha$, and thus $\alpha \in H_f(P)$.
\end{proof}

\begin{remark}
\emph{Consider now the case where $r(\alpha P)=r((\alpha -1)P)+1$ and $\mbox{Obstr}_0((\alpha-1)P)=\emptyset$. Take $A\in\mbox{Obstr}((\alpha-1)P)$ and write $A=A'+\beta P$, where $\supp(A')\not\ni P$ and $A'\geq 0$. Notice we have $0<\beta<\deg(A)=r(\alpha P)\leq \alpha$ in this case. The two relevant inequalities are $\Delta(f)+(\alpha-\beta)P-A'\geq 0$ and $\Delta(f)+(\alpha-\beta-1)P-A' \not\geq 0$. Whence we have $r((\alpha-\beta)P)=r((\alpha-\beta-1)P)+1$ and $\mbox{Obstr}_0((\alpha-\beta-1)P)\supseteq \{A'\}\neq\emptyset$. Thus $\alpha-\beta\in H_r^{red}(P)\subseteq H_f(P)$, whereas $\alpha\in H_r(P)$. By definition $H_r^{red}(P)\subseteq H_r(P)$.} \\
\end{remark}

We will see that $H_r^{red}(P)=H_r(P)$ for some graphs, and we conjecture that $H_r^{red}(P)=H_r(P)$ for any graph $G$. Taking a step towards proving the conjecture, we need another result from Baker and Norine, namely \cite[Proposition 3.1]{BN}. For $v\in A\subseteq V(G)$, let $\mbox{outdeg}_A(v)=|\{w\in V(G): vw\in E(G) \mbox{ and } w\in V(G)-A\}|$. A \emph{G-parking} function relative to a base vertex $v_0\in V(G)$ is an integer-valued function $f$ satisfying the conditions: 1) $f(v)\geq 0$ $\forall\; v\in V(G)-\{v_0\}$; 2) $\exists\; v\in A$ with $f(v)<\mbox{outdeg}_A(v)$ for every nonempty set $A\subseteq V(G)-\{v_0\}$. A divisor $D$ is said to be $v_0$-\emph{reduced} if the function defined by $v\mapsto D(v)$ is a G-parking function relative to $v_0$.\\

\begin{proposition}{\cite[Proposition 3.1]{BN}}\label{canon-rep}
Fix a base vertex $v_0\in V(G)$. Then, for every $D\in Div(G)$, there exists a unique $v_0$-reduced divisor $D'\in V(G)$ such that $D'$ and $D$ are linearly equivalent; i.e., $D'-D=\Delta(f)$ for some $f\in \mathcal{M}(G)$.
\end{proposition}

\begin{corollary}
Suppose $P\in V(G)$ has degree one or has a neighbor of degree one, then $H_r(P)=H_r^{red}(P)\subseteq H_f(P)$.
\end{corollary}

\begin{proof}
First, notice the condition obstructing a divisor $\alpha P$ of rank $k-1$ from gaining rank $k$ is preserved under linear equivalence of divisors: Given $E,F\in Div_+^k(G)$ and $E=F+\Delta(g)$ for some $g\in \mathcal{M}(G)$, the condition $\Delta(f)+\alpha P-E\not\geq 0$ is equivalent to the condition $\Delta(f-g)+\alpha P-F\not\geq 0$. \\

Now, for any $E\in\mbox{Obstr}(\alpha P)$, Proposition~\ref{canon-rep} guarantees that $E$ is linearly equivalent to a $v_0$-reduced divisor $F$. If $P$ has degree one, let $v_0$ be the unique neighbor of $P$. Since the map $v\mapsto F(v)$ is a G-parking function, we must have $0\leq F(P)<1=\mbox{outdeg}_{\{P\}}(P)$, and thus $F(P)=0$. Separately, let $v_0$ be a degree-one neighbor of $P$ if it exists. Notice that $\mbox{outdeg}_{(V(G)-\{v_0\})}(P)=1$, whereas $0\leq F(v)\not<\mbox{outdeg}_{(V(G)-\{v_0\})}(v)=0$ for $v\in V(G)-\{v_0,P\}$. Thus, we must have $F(P)=0$ for $F$ to be a G-parking function.

\end{proof}

Next, we define a family of functions that will be useful in determining certain elements of numerical semigroups.

\begin{definition}
Let $G$ be a graph with vertex set $V(G) = \{P_1, P_2, \ldots, P_n\}$. The indicator function $f_{P_i}$ is defined by
\[ f_{P_i}(P_j) = \begin{cases}
-1 & \text{if } i = j,\\
0  & \text{otherwise.}
\end{cases}\]
\end{definition}

\begin{remark}\label{degP}
\emph{Notice that $\Delta(f_P)=(\sum_{Q \in N(P)}Q)-(\deg P)P$; hence, we have $\deg P \in H_f(P)$.}
\end{remark}

Next, we show that $\deg(P)$ is the smallest nonzero element of $H_f(P)$, provided that the vertex $P$ is not a cut-vertex of $G$.

\begin{theorem}\label{min-degree}
Let $G$ and $G'=G-P$ be connected graphs. Then $\min\{a\in H_f(P)|a\neq 0\}=\deg(P)$.
\end{theorem}

\begin{proof}
Put $b=\min\{a\in H_f(P)|a\neq 0\}$. Define $W_f=\{w\in V(G)|f(w)=\min_{v\in V(G)}f(v)\}$ for any $f\in \mathcal{M}(G)$. We first show $b\geq \deg(P)$. So, suppose $\exists f$ with $\triangle(f)=A-cP$, where $0\leq A, 0<c<\deg(P)$, and $P\notin \mbox{Supp}(A)$. First, assume $W_f-\{P\}\neq \emptyset$. Take any $w_0\in W_f-\{P\}$. Then $\triangle_{w_0}(f)=\sum_{v\in N(w_0)}(f(w_0)-f(v))\leq 0$. Since $f(w_0)\leq f(v)\,\forall v\in N(w_0)$ and $\triangle_{w_0}(f)\geq 0$, we must have $f(w_0)=f(v)\,\forall v\in N(w_0)$. Since $G'$ is connected, $\exists w_1\in N_{G'}(w_0)$ and $f$ must be locally constant at $w_1$ (i.e., $f(w_1)=f(v)$ if $v\in N_{G'}(w_1)$) by the same argument; propagating along the vertices of $G'$ in this way, we conclude $f$ must be constant on $G'$ as $\triangle_{w}(f)\geq 0\,\forall w\in V(G')$. Thus, $\forall w\in V(G')$, we have $f(w)=\min_{v\in V(G)}f(v)$, implying $f(P)-f(w)\geq 0$ for every $w\neq P$; this contradicts the assumption that $\triangle_P(f)=\sum_{w\in N_G(P)} (f(P)-f(w))=-c<0$.\\

Thus, we must have $W_f=\{P\}$; i.e., $f(v)>f(P)\, \forall v\in V(G')$. WLOG, let $f(P)=0$. Then $-\triangle_P(f)=-\sum_{v\in N_{G}(P)}(f(P)-f(v))=c<\deg(P)$. This means $\sum_{v\in N_G(P)}(f(v)-f(P))<\deg(P)$, which is impossible; hence $b\geq \deg(P)$.
By Remark~\ref{degP}, $b=\deg(P)$.
\end{proof}

\begin{porism}\label{unique-min}
Let $G$ and $G-P$ be connected graphs, and let $\triangle(f)=A-cP$ with $A\geq 0$ and $P\notin\mbox{supp}(A)$. Then $f(P)<f(v)\,\forall v\in V(G-P)$.
\end{porism}

\begin{remark}\label{need G-P be connecetd}
\emph{The connectedness of $G-P$ is necessary for both Theorem~\ref{min-degree} and Porism~\ref{unique-min}, as the following example shows. Let $P$ be a specific vertex in a clique $\Omega_m$ of order $m\geq 2$, and join $P$ to a new vertex $v$ by an edge. Consider the function $f$ defined by $f(v)=1$ and $f(u)=0$ for $u\neq v$. Clearly, $\Delta(f)=v-P$, whereas $\deg(P)=m$; nor is $P$ the unique minimum of $f$. Note that $G-P$ has two components.}
\end{remark}

\begin{proposition}\label{multi-comp G-P} Let $P\in V(G)$. Let $G_1,\cdots,G_m$, where $m\geq 2$, be the connected components of $G-P$. Then $\deg_{G_i}P$, the number of edges incident with $P$ in $G[V(G_i)\cup\{P\}]$ (the subgraph induced by $P$ and $V(G_i)$), belongs to $H_f(P)$ for $1\leq i\leq m$.
\end{proposition}

\begin{proof}
Consider the function $f$ defined on $V(G)$ by $f(u)=1$ for $u\in V(G_i)$ and $f(u)=0$ otherwise. Then $f$ has pole only at $P$ of order $\deg_{G_i}$.
\end{proof}

Combining Theorem~\ref{min-degree} and Proposition~\ref{multi-comp G-P}, we obtain the following result.

\begin{corollary}\label{min-degree-connectivity}
Let $P\in V(G)$. Then $\min\{a\in H_f(P)|a\neq 0\}=\deg(P)$ if and only $P$ is not a cut-vertex of $G$.
\end{corollary}

Recall that edge connectivity of a graph $G$, denoted by $\lambda(G)$, is the minimum number of edges of $G$ whose deletion disconnects $G$. We say that $G$ is $(k+1)$-edge connected if and only if no set of $k$ edges disconnects $G$. Baker and Norine in~\cite{BN} defined, for any vertex $P\in V(G)$ and any $k\in \mathbb{Z}^+$, the Abel-Jacobi map $S_P^{(k)}: Div_+^k(G)\rightarrow Jac(G)$, where $D \mapsto [D-kP]$ (here, $[A]$ denotes the class in $Jac(G)$ of the divisor $A$), and they proved the following result.

\begin{theorem}\cite{BN} The Abel-Jacobi map $S^{(k)}_P$ is injective if and only if the graph is $(k+1)-$edge connected.
\end{theorem}

\begin{corollary} \label{min-connectivity}
Let $P$ be a vertex of a graph $G$. Then $\lambda(G) \leq \min \left\{ \alpha \in H_f(P) : \alpha \neq 0 \right\}$.
\end{corollary}

\begin{proof}
Put $\lambda =\lambda(G)$, and consider $1\leq \gamma \leq \lambda-1$. Note that $G$ is $\left(\gamma + 1\right)$-edge connected. Suppose there exists $f \in \mathcal M(G)$ such that $\Delta(f)=A-\gamma P$ with $A \geq 0$ and $P \notin \supp A$. Then $A \in Div_+^{\gamma}G$, since a principal divisor has degree zero. Notice $S_P^{(\gamma)}(A)=[A-\gamma P]=[0]=[\gamma P -\gamma P] = S_P^{(\gamma)}(\gamma P)$, since $[\Delta(f)]=[0]$ in $Jac(G)$. By the injectivity of $S_P^{(\gamma)}$, $A=\gamma P$, which is a contradiction. Thus, $\gamma \in G_f(P)$ for all $\gamma\in \{1,\cdots,\lambda-1\}$.
\end{proof}

\begin{remark}
\emph{The (vertex) connectivity of a graph $G$, denoted by $\kappa(G)$, is the minimum number of vertices of $G$ whose deletion disconnects $G$. Denote by $\delta(G)$ the minimum degree of a graph $G$ and by $\lambda_2(G)$ the second smallest eigenvalue of the Laplacian of $G$. Then, it is well known that $\delta(G)\geq\lambda(G)$ and $\lambda_2(G)\leq\kappa(G)\leq\lambda(G)$ (see, for example, \cite{Bollobas}). We note that Corollary~\ref{min-connectivity} extends the chain of inequalities from the opposite side of $\lambda_2(G)$, with a parameter also stemming from the Laplacian of $G$.}
\end{remark}

\section{On Weierstrass semigroups of some graph families} \label{exs}

In this section, we determine $H_f(P)$ and consider its connection with $H_r(P)$ for trees, unicyclic graphs, complete graphs.\\

\begin{proposition}
If $G$ is a tree, then $H_f(P)=H_r(P)=H_r^{red}(P)=\N$ for any vertex $P$ of $G$.
\end{proposition}

\begin{proof}
That $H_f(P)=\N$ follows immediately from Proposition~\ref{multi-comp G-P} and Corollary~\ref{min-degree-connectivity}, since a vertex $P$ on a tree is either an end-vertex (thus of degree one) or a cut-vertex joined by one edge to each of the components of $G-P$. That $H_r(P)=H_r^{red}(P)=\N$ can be seen as follows. Since $g=|E(G)|-|V(G)|+1=0$, $\deg(K-\alpha P)=-2-\alpha<0$ and thus $r(K-\alpha P)=-1$ for any $\alpha\in \mathbb{N}$. Hence $r(\alpha P)=\deg(\alpha P)+1-g+r(K-\alpha P)=\deg(\alpha P)=\alpha$ for $\alpha\in \mathbb{N}$. Notice that any divisor $E\in Div_+^{\alpha+1}(G)$ obstructs $\alpha P$ from attaining the rank of $\alpha+1$: no function $f$ satisfies $\Delta(f)\geq E-\alpha P$, since $\deg(\Delta(f))=0$ and $\deg(E-\alpha P)=1$. Thus, $(\alpha+1)Q$ for $Q\in V(G)-\{P\}$ is an obstructing divisor for $\alpha P$; i.e., $\mbox{Obstr}_0(\alpha P)=\{A\in\mbox{Obstr}(\alpha P):\supp(A)\not\ni P\}\neq \emptyset$.\hfill
\end{proof}

\begin{proposition}\label{unicyclic}
Let $G$ be a unicyclic (genus one) graph. Then $H_r(P)=\N-\{1\}$ for any $P\in V(G)$, and
$H_f(P)=\left\{\!\!\begin{array}{ll}
\N-\{1\} & \mbox{if $P$ is a degreee two vertex lying on the unique cycle of $G$,} \\
\N & \mbox{otherwise.}
\end{array}
\right.$
\end{proposition}

\begin{proof}
Since $g=1$, we have $H_r(P)=\mathbb{N}-\{1\}$ (for any $P\in V(G)$) by Lemma~\ref{hrgaps}. If $P$ is a cut-vertex of $G$, then $\min_{i}(\deg_{G_i}(P))=1$, where $\deg_{G_i}(P)$ denotes the degree of $P$ within the $i$-th component subgraph $G[V(G_i)\cup \{P\}]$. Thus, $H_f(P)=\N$ by~Proposition~\ref{multi-comp G-P}.\\

If $G-P$ is connected for $P\in V(G)$, then either $P$ has $\deg(P)=1$ or $P$ is a degree-two vertex on the unique cycle $C$ of $G$. If $\deg(P)=1$, then Theorem~\ref{min-degree} yields $H_f(P)=\mathbb{N}$. If $P$ is a degree-two vertex on $C$, then Theorem~\ref{min-degree} yields $2=\min(H_f(P)-\{0\})$; we will show that $3\in H_f(P)$, and then $H_f(P)=\langle 2,3\rangle=\mathbb{N}-\{1\}$ in this case.\\

Let the vertices on the unique cycle $C$ of $G$ be cyclically labelled $v_1, \ldots v_k$; note $k\geq 3$, and let $E(C)$ be the set of edges of $C$. By \emph{the tree rooted at} $v\in V(C)$, we mean the component of $G-E(C)$ which contains $v$. WLOG, let $P=v_1$ be a vertex of degree two. Consider the function $f$ on $V(G)$ defined by $f(v_1)=0$, $f(u)=1$ for any vertex $u$ belonging to the tree rooted at $v_k$, and $f(u)=2$ otherwise. Then $\Delta(f)=E-3P$, where $E\geq 0$ and $P\notin \supp (E)$, and thus we have $3\in H_f(P)$.\hfill
\end{proof}

\begin{remark}
\emph{Proposition~\ref{unicyclic} shows that $H_r(P)\subseteq H_f(P)$ at any point $P$ on a unicyclic graph $G$. But, can we see the containment $H_r(P)\subseteq H_f(P)$ in an alternate way, through the Riemann-Roch lens?}\\

\emph{To this end, let $D=0P=0$ and $g=1$; we have from R-R the equation $r(0)=0+1-1+r(K-0)$. But, $r(0)=0$ by definition for the rank of a divisor; this yields $r(K)=0$, which means $K$, besides having degree zero, equals zero as a divisor in $Jac(G)$. Thus, noting $r(-P)=-1$ by definition, we have $r(P)=1+r(-P)=0=r(0P)$, and hence $1\in G_r(P)$. With $r(K-2P)=-1$, we find $r(2P)=1$. Now, $r(P)=0$ and $r(2P)=1$ imply the existence of an $E\in Div_+^1(G)$ and an $h\in M(G)$ such that $\Delta(h)\not\geq E-P$, whereas $\Delta(h)\geq E-2P$. Clearly, $E$ cannot be $P$. So, $\Delta(h)=E'-2P$ with $E'\in Div_+^2(G)$ and $P\notin \supp(E')$. This gives us $2\in H_f(P)$.} \\

\emph{To see that $3\in H_r(P)$ also belongs to $H_f(P)$, consider $E\in Div_+^2(G)$ such that $\Delta(f)\not\geq E-2P$ $\forall f\in M(G)$. If $P\notin\supp(E)$, then $3\in H_r^{red}(P)\subseteq H_f(P)$ and we are done. Otherwise, noting $E\neq 2P$, put $E=Q+P$, for $Q\neq P$. Applying R-R to $E$, we have $r(P+Q)=1$; this means that for any $Q'\in V(G)-\{P,Q\}$, there is a $f'\in M(G)$ such that $\Delta(f')\geq Q'-(P+Q)$. Thus, $\Delta(f')=Q''+Q'-(Q+P)$. Notice the condition $(\forall f\in M(G)$, $\Delta(f)\not\geq E-D)$ holds when $E$ is replaced by $E'=E+\Delta(h)$. So, Let $E'=E+\Delta(f')=Q+P+[Q''+Q'-(Q+P)]=Q''+Q'$. We cannot ensure $Q''\notin\{P,Q,Q'\}$. However, if $Q''\neq P$, then we've shown that $E'\in\mbox{Obstr}_0(2P)$ and thus $3\in H_r(P)\cap H_f(P)$.}\\

\emph{If $Q''=P$, we can switch to $H_f(Q)$. It is significant that $r(\alpha P)$, hence $H_r(P)$, depends only on $\alpha$ and not on the vertex $P$ in question; i.e., $H_r(P)=H_r(Q)$. Now, if there is $h\in M(G)$ such that $\Delta(h)\geq Q+P-2Q=P-Q$, then $\Delta(h)=P-Q$. Then, $\Delta(-h)=-\Delta(h)=Q-P=Q+P-2P\geq Q+P-2P$, contradicting the earlier assumption on the obstructing divisors for $2P$. Thus, $E=P+Q$ obstructs the divisor $2Q$ from attaining the rank $r(2Q)+1=2$. Since we are in the case $Q''=P$, we see that $E'=E+\Delta(f')=P+Q'$, where $Q\notin\{P,Q'\}$, also obstructs $2Q$; this implies that $3\in H_r(Q)\cap H_f(Q)$. Note further that if there is a graph isomorphism taking vertex $P$ to vertex $Q$, then we have $H_f(Q)=H_f(P)$.}\hfill $\Box$
\end{remark}

A bit more insight into the interplay between $H_r(P)$ and $H_f(P)$ is provided in the next example. To this end, the notion of the ``chip-firing game" (also known as the ``dollar game" or the ``sandpile toppling game") discussed in~\cite{BN} is helpful. Given a divisor $D$, assign $D(v)$ chips (or dollars) to each vertex $v$ of a graph $G$. Each \emph{move} of the game consists of a vertex either taking one chip from each of its neighbors or giving one chip to each of its neighbors. The following result is informative.

\begin{lemma}\cite{BN}\label{chip-firing} Two divisors $D$ and $D'$ on a graph $G$ are linearly equivalent if and only if there is a sequence of moves in the chip-firing game which transforms the configuration corresponding to $D$ into the configuration corresponding to $D'$.
\end{lemma}

\begin{remark}
\emph{When the rank of a divisor cannot be readily determined from the Riemann-Roch theorem of Baker and Norine, the ideas developed herein may be able to help, as the next example will show.}
\end{remark}

\begin{example}
\emph{Given a cycle on $n-1$ vertices and a new vertex $P$, the wheel graph $W_{1,n-1}$ on $n$ vertices is formed by drawing an edge from each vertex of the cycle to $P$. Notice that here $n\geq 4$ and the genus (or cycle rank) of $W_{1, n-1}$ is $n-1$. Since $\deg(P)=n-1$, by Theorem~\ref{min-degree}, $\min H_f(P)=n-1$. If $H_r(P)\subseteq H_f(P)$ as conjectured, then we must have $r(0P)=r(P)=\ldots=r((n-2)P)=0$ and $r((n-1)P)=0$ or $r((n-1)P)=1$.} \\

\emph{It is easy to see that $r((n-1)P)\geq 1$. Let $E\in Div_+^1(G)$; then $E=Q_0\in V(W_{1,n-1})$. If $Q_0=P$, then any constant function $f$ satisfies $\Delta(f)\geq E-(n-1)P=-(n-2)P$. If $Q_0\neq P$, then the function $f$, defined by $f(P)=0$ and $f(Q)=1$ otherwise, yields $\Delta(f)=\sum_{Q\neq P}Q-(n-1)P\geq E-(n-1)P$. Alternatively, with one chip-firing at vertex $P$, Lemma~\ref{chip-firing} tells us that $D=(n-1)P$ is linearly equivalent to $D'=\sum_{Q\neq P}Q$; i.e., $D-D'$ is the divisor of a function. } \\

\emph{On the other hand, if $r((n-2)P)\geq 1$, then, taking $E=Q_0\neq P$ to be a fixed vertex on the cycle, there must exist a non-constant function $f$ with $\Delta(f)\geq Q_0-(n-2)P$. This means that $\Delta(f)=\left(Q_0+\sum_{i=1}^{n-3}Q_i\right)-(n-2)P$, with $Q_0\neq P$. By Lemma~\ref{chip-firing}, the existence of $\Delta(f)$ is equivalent to being able to transform the configuration given by $D=(n-2)P$ to the one given by $D'=Q_0+\sum_{i=1}^{n-3}Q_i$ through a sequence of chip-firing moves. After some trial and error, we may be forgiven for becoming convinced that this is an impossible task: take, for example,  $W_{1,4}$; starting with $D$, one returns to $D$ after five moves, with no effective divisor in between.}\\

\emph{In order to prove that the task is indeed impossible, we might first try R-R and find, from $r((n-2)P)=(n-2)+1-(n-1)+r(K-(n-2)P)$, that R-R simply yields $0\leq r((n-2)P)=r(K-(n-2)P)$. Since $\deg(K-(n-2)P)=2(n-1)-2-(n-2)=n-2>0$ for $n\geq 4$, the rank of $(K\!-\!(n-2)P)$ is not immediately clear. Again, assume $r((n-2)P)\geq 1$ and take a vertex $Q_0\neq P$; we find that the $f\in M(G)$ in the preceding paragraph has $\triangle(f)=A+Q_0-\alpha P$, where $A\geq 0,\; P\notin\supp(A+Q_0)$, and $1\leq\alpha\leq n-2$. However, Theorem~\ref{min-degree} rules out the existence of this $f$, as it would mean $H_f(P)\ni\alpha$, where $1\leq\alpha\leq n-2$.}\\

\emph{Thus, we proved $r((n-2)P)=0$ and that the aforementioned task is indeed impossible, since $(n-2)P$ is the only effective divisor in its linear equivalence class. This also implies $r((n-1)P)=1$, since we showed that $r((n-1)P)\geq 1$.} \hfill$\Box$
\end{example}

We look at complete graphs as another example.

\begin{proposition}
Let $P$ be a vertex of the complete graph $K_n$ with $n\geq 2$. Then $H_f(P)=\langle n-1,n\rangle$.
\end{proposition}

\begin{proof}
Since $K_n$ and $K_n-P$ are connected, a function $f$ associated with $H_f(P)$ (i.e., having a unique pole at $P$) has a unique minimum at $P$ by Porism~\ref{unique-min}. Choose $Q\in V(K_n-P)$ such that $f(Q)=\min\{f(v)|v\in V(K_n-P)\}$. WLOG, let $f(P)=0$ and $f(Q)=a$; note $a\in \mathbb{Z}^+$. Then, we have

\begin{eqnarray}
0\leq\triangle_Q(f)=\sum_{v\in N(Q)}f(Q)-f(v)=\deg(Q)f(Q)-\sum_{v\in V(K_n-Q)}f(v) \nonumber\\
=(n-1)a-\sum_{i=1}^{n-2}(a+\alpha_i) \ (\mbox{where each}\, \alpha_i\geq 0) \nonumber\\
=(n-1)a-(n-2)a-\sum_{i=1}^{n-2}\alpha_i \nonumber\\
=a-\sum_{i=1}^{n-2}\alpha_i\,;\nonumber\\
\mbox{thus, we have } a\geq \sum_{i=1}^{n-2}\alpha_i\, .\label{a-alphas}
\end{eqnarray}
Now,
\begin{eqnarray*}
-\triangle_P(f)=-[\deg(P)f(P)-\sum_{v\in N(P)}f(v)] \\
=\sum_{v\in N(P)}f(v)=a+\sum_{i=1}^{n-2}(a+\alpha_i)=(n-1)a+\sum_{i=1}^{n-2}\alpha_i \, .
\end{eqnarray*}

Using Inequality~(\ref{a-alphas}) and the fact that each $\alpha_i\in\mathbb{N}$, we have $\displaystyle (n-1)a\leq -\triangle_P(f)=(n-1)a+\sum_{i=1}^{n-2}\alpha_i\leq na$; this means that $-\triangle_P(f)=a(n-1)+i$ for some integer $i$ with $0\leq i\leq a$.
By the chain of equalities $a(n-1)+i=a(n-1)-in+i+in=a(n-1)-i(n-1)+in=(a-i)(n-1)+in$, we find $H_f(P)\subseteq\langle n-1,n\rangle$.\\

To see $\langle n-1, n\rangle\subseteq H_f(P)$, let $f_P$ and $f_Q$ be the two indicator functions at $P$ and $Q$, and let $g=f_P-f_Q$. Then $\triangle(f_P)=[\sum_{v\neq P}v]-(n-1)P$ and $\triangle(g)=nQ-nP$.

\end{proof}

\begin{remark}
\emph{Since the genus (cycle rank) of $K_n$ is ${n\choose 2}-n+1$, which equals $|\N \setminus \langle n-1,n\rangle|$, we have $|\N\setminus H_r(P)|=|\N\setminus H_f(P)|$. Thus, if $H_r(P)\subseteq H_f(P)$ in this case, then we could conclude that $H_r(P)=\langle n-1, n \rangle$ for any vertex $P$ of $K_n$.}  \\
\end{remark}

As implied by Lemma~\ref{hrgaps} and Proposition~\ref{multi-comp G-P}, $H_f(P)\setminus H_r(P)$ can be arbitrarily large. This is seen explicitly in the next result.\\

\begin{proposition}\label{lggap}
For every $n \in \Z^+$,  there is a graph $G$ with vertex $P$ so that $|H_f(P)\setminus H_r(P)| = n$.
\end{proposition}

\begin{proof}
Let graphs $G_1$ and $G_2$ have genera $g_1$ and $g_2$, respectively, and let $P$ be a new vertex. The graph $G$ is formed from graphs $G_1$, $G_2$ and the vertex $P$ by joining $P$ with an edge to exactly one vertex of $G_i$ for each $i\in\{1,2\}$; see Figure~\ref{gr_bridge} for an illustration. To see that $H_f(P)=\N$, consider the function $f \in \mathcal{M}(G)$ given by
 $$
 f(v)=\begin{cases}
 1 & \text{ if } v \in V(G_1) \\
 0 & \text{ otherwise}.
 \end{cases}
 $$
 Notice $\Delta(f)=Q-P$, where $Q \in V(G_1)$ is adjacent to $P$. However, $G$ has genus $g_1 + g_2$ and thus $|G_r(P)|=g_1 + g_2$, while $H_f(P)$ has no gaps. As a result, $|H_f(P)\setminus H_r(P)| = g_1 + g_2$, which may be any $n\in\Z^+$.
\end{proof}

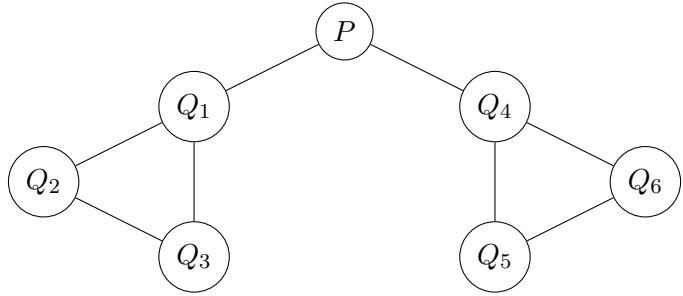
\begin{figure}[h]
\centering
\begin{tikzpicture}
\tikzstyle{every node}=[draw,shape=circle];
\path (2,0) node (v1) {$Q_1$};
\path (0,-1) node (v2) {$Q_2$};
\path (2,-2) node (v3) {$Q_3$};
\path (6,0) node (v4) {$Q_4$};
\path (6,-2) node (v5) {$Q_5$};
\path (8,-1) node (v6) {$Q_6$};
\path (4, 1) node (v7) {$P$};
\draw (v1) -- (v2)
			(v1) -- (v3)
			(v1) -- (v7)
			(v2) -- (v3)
			(v4) -- (v7)
			(v4) -- (v6)
			(v4) -- (v5)
			(v5) -- (v6);
\end{tikzpicture}
\caption{$|\N\setminus H_r(P)|=2$ and $|\N\setminus H_f(P)|=0$}
\label{gr_bridge}
\end{figure}

\section{Conclusion} \label{conclusion}

Our work prompts several questions. First, is it true that $H_r(P)\subseteq H_f(P)$ for any vertex $P$ on any graph $G$ of order at least two? We conjecture that it is. Second, which numerical semigroups arise as a Weierstrass semigroup of a vertex of a finite graph? The analogous problem for points on curves is a deep one, first suggested by Hurwitz \cite{Hurwitz}. Nearly 100 years later, Buchweitz \cite{Buch} proved that not every numerical semigroup is the Weierstrass semigroup of a point on a curve and defined what is now called the Buchweitz Criterion. This problem was further addressed in \cite{eh, Komeda} (see also references therein) and more recently \cite{Kaplan} but remains open. What can be said about the structure of $H_r(P)$? The analogous set for points on curves (defined appropriately) has the property that $H_r(P)=H_f(P)$. However, we see that this fails dramatically for vertices on finite graphs, as demonstrated in Proposition~\ref{lggap}. This leaves one to consider what more can be said about $H_r(P)$. Of course, one may study Weierstrass semigroups of vertices on certain families of graphs. In particular, one may consider covers of graphs and associated semigroups, as has been done for coverings of curves \cite{KW, op2}.

\paragraph*{Acknowledgment}
The authors hereby express their sincere appreciation to the referees for their very careful reading of and helpful comments on an earlier draft of the paper, which has been improved as a result.\\

\end{document}